\documentclass[graybox]{svmult}
\usepackage{amsmath}
\usepackage{mathrsfs}
\usepackage{amssymb}
\usepackage{cite}
\usepackage{hyperref}

\usepackage{mathptmx} 
\usepackage{helvet} 
\usepackage{courier} 
\usepackage{type1cm} 
\usepackage{makeidx} 
\usepackage{graphicx} 
\usepackage{multicol} 
\usepackage[bottom]{footmisc}

\makeindex 

\makeatother

\begin{document}

\title*{SOLO FTRL algorithm for production management with transfer prices}
\author{Dmitry B. Rokhlin and Gennady A. Ougolnitsky}
\institute{D.B. Rokhlin \at I.I.\,Vorovich Institute of Mathematics, Mechanics and Computer Sciences of Southern Federal University and Regional Scientific and Educational Mathematical Center of Southern
Federal University \email{dbrohlin@sfedu.ru}
\and G.A. Ougolnitsky \at I.I.\,Vorovich Institute of Mathematics, Mechanics and Computer Sciences of Southern Federal University \email{gaugolnickiy@sfedu.ru}}
%
%
\maketitle

\abstract{We consider a firm producing and selling $d$ commodities, and consisting from $n$ production and $m$ sales divisions. The firm manager tries to stimulate the best division performance by sequentially selecting internal commodity prices (transfer prices). In the static problem under general strong convexity and compactness assumptions we show that the SOLO FTRL algorithm of Orabona and P{\'a}l (2018), applied to the dual problem, gives the estimates of order $T^{-1/4}$ in the number $T$ of iterations for the optimality gap and feasibility residuals. This algorithm uses only the information on division reactions to current prices. It does not depend on any parameters and requires no information on the production and cost functions. The results of similar nature are obtained for the dynamic problem, where these functions depend on an i.i.d. sequence of random variables. We present two computer experiments with one and two commodities. In the static case the transfer prices and the difference between the supply and demand demonstrate fast stabilization. In the dynamic case the same quantities fluctuate around equilibrium values after a short transition phase.}


\section{Introduction} 
Divisions of a large firm frequently act as independent profit maximizing agents. In this paper we consider a firm with several production and sales divisions. The firm management can settle internal prices for the produced commodities to coordinate the division activities. The production (resp., sales) division should sell (resp., buy) commodities at these \emph{transfer prices}, which are valid only within the firm. The aim of such pricing system is to stimulate agent behavior, resulting in the profit maximization of the whole firm. The problem of constructing such pricing policies was formulated in 1950s \cite{Hirshleifer1956}. Now the literature on transfer pricing has is quite rich. We refer to the reviews \cite{Gox2006}, \cite{Baldenius2009}, \cite[Chap. 13]{Labro2019}. The most influential papers related to the topic of transfer pricing were recently collected in \cite{Eden2019}. 

In \cite{Schuster2015} three major types of transfer prices are highlighted: market-based, cost-based and negotiated. Among cost based transfer prices the marginal ones are mathematically most natural, since under the convexity assumption they stimulate the firm-optimal solution. However, this pricing method is frequently criticized as follows: to get the transfer price one should find the firm-optimal solution, which is impossible in practice. Moreover, if for some reason such solution is known, then its related components can simply be communicated to each division by the firm manager. So, there is no need in any pricing system.

In the present paper the revenue and cost functions of the divisions are assumed to be unknown. Instead, the firm manager gets the division responses for the transfer prices. So, the transfer pricing is the only way to indirectly access the agent revenue and cost function. Is it possible, using only the agent responses, to find transfer prices  stimulating approximately firm-optimal solutions? Roughly speaking, we give an affirmative answer with some quantitative estimates. 

Dual gradient descent algorithms, which we utilize for this purpose, are widely used for resource allocation in communication networks \cite{Low1999,Beck2014,Wu2019,Rokhlin2021}. In these algorithms the link prices are updated on the basis of user (or processor) reactions to current prices. The firm profit maximization corresponds to the network utility maximization problem, formulated in \cite{Kelly1998}. Note, however, that usually the mentioned algorithms need some information concerning the user utilities, like Lipschitz constants or strong convexity parameters. In the present paper we try to avoid using such information (note that the intention of \cite{Rokhlin2021} was similar). To this end, we apply the SOLO FTRL algorithm of Orabona, P{\'a}l \cite{Orabona2018} to the dual problem. The resulting updating rule for the transfer price does not depend on any parameters, and requires no information on the production and cost functions.

The paper is organized as follows. In Section \ref{sec:2} we formally describe the static problem, where the revenue and cost functions of the divisions are unknown but fixed. Under convexity and compactness assumptions we show that the firm-optimal solution is stimulated by some transfer price vector (Theorem \ref{th:1}). This is a simple and standard result.
Then, similarly to \cite{Beck2014}, we apply Nesterov's fast gradient descent algorithm to the dual problem and obtain the estimates of order $T^{-1}$ in the number $T$ of iterations for optimality and feasibility residuals (Theorem \ref{th:2}). After this we apply the mentioned SOLO FTRL algorithm. Its convergence rate is slower: the same residuals are of order $T^{-1/4}$ (Theorem \ref{th:3}), but it uses only the information on division reactions to current prices. 

In Section \ref{sec:3} we consider the dynamic problem, where the revenue and cost functions depend on a sequence of i.i.d. random variables. To evaluate the quality of transfer prices we use the average regret: the quantity which is standard in the online learning theory. The main result of the paper (Theorem \ref{th:4}) states that the price sequence generated by the SOLO FTRL algorithm  ensures no-regret learning with respect to the best possible plan sequence, and the equilibrium between the supply and demand is satisfied on average. It also states that the stochastic bounds of order $T^{-1/4}$. Note that the best possible plan sequence is a rather strong comparator.

In Section \ref{sec:4} we present two computer experiments with one and two commodities. In the static case the transfer prices and the difference between the supply and demand demonstrate fast stabilization. In the dynamic case the same quantities fluctuate around equilibrium values after a short transition phase.

\section{Static problem} \label{sec:2}
Consider a firm consisting from $n$ production and $m$ sales divisions. There are $d$ commodities produced by each production division. The same commodities are sold by each sales division. In this section we consider the static problem where $f_i:X_i\mapsto\mathbb R_+$, $i=1,\dots,m$ are the revenue functions of the sales divisions,  $g_i:Y_i\mapsto\mathbb R_+$, $i=1,\dots,n$  are the cost functions of the production divisions, and $X_i$, $Y_i$ are some convex subsets of $\mathbb R^d_+=\{x\ge 0:x\in\mathbb R^d\}$. A vector $x_i\in X_i$ describes the amounts of commodities to be sold by $i$-th sales division, and $y_i\in Y_i$ describes the amounts of commodities to be produced by $i$-th production division. Under an unrealistic assumption that the functions $f_i$, $g_i$ are completely known, the firm manager can solve the profit maximization problem
\begin{align}
 F(x,y)=\sum_{i=1}^m f_i(x_i)-\sum_{i=1}^n g_i(y_i)\to\max_{(x,y)\in S}, \label{1.1}\\
  S=\left\{(x,y)\in Z: \sum_{i=1}^m x_i=\sum_{j=1}^n y_j\right\}, \quad Z=\prod_{i=1}^m X_i\times\prod_{j=1}^n Y_j, \label{1.2}
\end{align}  
and ensure an optimal firm performance by assigning to each division the related component of an optimal solution. The constraint $(x,y)\in S$ requires that the total production equals to the total sales in each commodity: an equilibrium between the supply and demand at the firm level.

Under convexity assumptions there is also a more economic way to achieve the same goal. The firm can announce the commodity \emph{transfer price} vector $\lambda_t\in\mathbb R^d_+$ with the obligation to buy the commodities at these prices from the production divisions, and sell them to the sales divisions. Let us introduce the standing assumptions that will be used in the rest of the paper.

Recall that a function $f:A\mapsto\mathbb R$, defined on a convex set $A$ is called $\sigma$-strongly convex (with some $\sigma>0$)  if
\[ f(\alpha x+(1-\alpha)y)\le\alpha f(x)+(1-\alpha)f(y)-\frac{\sigma}{2}\alpha(1-\alpha)\|x-y\|^2\] 
for all $x,y\in A$, $\alpha\in [0,1]$. By $\|\cdot\|$ we always denote the usual Euclidean norm.

\begin{description}
\item[\bf{Assumption 1.}]  \label{as:1} The sets $X_i$, $Y_i$ are convex, compact, and 
\[ [0,\varepsilon]^d\subset X_i \subset [0,c]^d,\quad [0,\varepsilon]^d\subset Y_i \subset [0,c]^d \]
with some $\varepsilon>0$, $c>0$.
\item[\bf{Assumption 2.}]  \label{as:2} The functions $f_i:X_i\mapsto\mathbb R_+$ (resp., $g_i:Y_i\mapsto\mathbb R_+$) are $\sigma_i'$-strongly concave (resp., $\sigma_i''$-strongly convex), non-decreasing in each argument, and $f_i(0)=g_i(0)=0$.
\item[\bf{Assumption 3.}]  \label{as:3} The functions $f_i$ (resp., $g_i$) are $K_i'$-Lipschitz (resp., $K_i''$-Lipschitz).
\end{description}

The optimal division (agent) reactions are uniquely defined by
\begin{align}
\widetilde x_i(\lambda)&=\arg\max_{x_i\in X_i}(f_i(x_i)-\langle\lambda,x_i\rangle),\quad i=1,\dots,m,\label{1.2A}\\
\widetilde y_i(\lambda)&=\arg\max_{y_i\in Y_i}(\langle\lambda,y_i\rangle-g_i(y_i)),\quad i=1,\dots,n, \label{1.2B}
\end{align}
where $\langle a,b\rangle=\sum_{i=1}^d a_i b_i$ is the usual scalar product. We will say that the plan $\widetilde z(\lambda)=(\widetilde x(\lambda),\widetilde y(\lambda))$ is stimulated by the transfer price vector $\lambda$. The next elementary result shows that the problem (\ref{1.1}), (\ref{1.2}) is equivalent to finding a price vector, stimulating an equilibrium.

\begin{theorem} \label{th:1} An admissible point $z^*=(x^*,y^*)\in S$ is an optimal solution of the problem  (\ref{1.1}), (\ref{1.2}) if and only if it is stimulated by some transfer price vector $\lambda^*\in\mathbb R^d_+$: $z^*=\widetilde z(\lambda^*)$.
\end{theorem}
\begin{proof} Consider the Lagrange function
\begin{align} \label{1.3}
L(x,y,\lambda)=F(x,y)+\sum_{i=1}^n\langle\lambda,y_i\rangle-\sum_{i=1}^m\langle\lambda,x_i\rangle.
\end{align}
If $\widetilde z(\lambda^*)\in S$, then it is an optimal solution of  (\ref{1.1}), (\ref{1.2}):
\begin{align}
F(\widetilde z(\lambda^*)) &=L(\widetilde z(\lambda^*),\lambda^*)\nonumber\\
&=\sum_{i=1}^m(f_i(\widetilde x_i(\lambda^*))-\langle\lambda^*,\widetilde x_i(\lambda^*)\rangle)+\sum_{i=1}^n(\langle\lambda^*,\widetilde y_i(\lambda^*)\rangle-g_i(\widetilde y_i(\lambda^*)))\nonumber\\
&\ge \sum_{i=1}^m(f_i(x_i)-\langle\lambda^*,x_i\rangle)+\sum_{i=1}^n(\langle\lambda^*,y_i\rangle-g_i(y_i))\nonumber\\
&\ge F(x,y)+\left\langle\lambda^*,\sum_{i=1}^n y_i-\sum_{i=1}^m x_i\right\rangle
=F(x,y),\quad (x,y)\in S. \label{1.4}
\end{align}

Conversely, if $z^*=(x^*,y^*)$ is an optimal solution of the mentioned problem, then by a version of the Karush-Kuhn-Tucker theorem \cite[Corollary 28.3.1]{Rockafellar1970}, there exists a vector $\lambda^*\in\mathbb R^d$ such that $(z^*,\lambda^*)$ is a saddle point of the Lagrange function:
\[ L(x^*,y^*,\lambda)\ge L(x^*,y^*,\lambda^*)\ge L(x,y,\lambda^*),\quad (x,y)\in\prod_{i=1}^m X_i\times\prod_{j=1}^n Y_j,\quad \lambda\in\mathbb R^d.\]
From the right inequality, which is similar to (\ref{1.4}), it follows that $z^*=\widetilde z(\lambda^*)$. If $\lambda^*$ has a negative component $\lambda_j^*<0$, then 
\[ \widetilde x_{i,j}(\lambda^*)\ge\varepsilon,\ i=1,\dots,m,\quad \widetilde y_{i,j}(\lambda^*)=0,\ i=1,\dots,n\]
for the same components $\widetilde x_{i,j}$, $\widetilde y_{i,j}$ of the vectors  $\widetilde x_i$, $\widetilde y_i$, since the functions $f_i$, $g_i$ are non-decreasing in each variable. But this contradicts to the equilibrium condition (in commodity $j$), which $(x^*,y^*)$ should satisfy. Thus, $\lambda^*\ge 0$. The proof is complete. $\square$
\end{proof}

The price vector $\lambda^*$, mentioned in Theorem \ref{th:1}, is an optimal solution of the dual minimization problem 
\begin{align} \label{1.5}
G(\lambda):=\sup_{(x,y)\in X\times Y} L(x,y,\lambda)=\sum_{i=1}^m v_i(\lambda)-\sum_{i=1}^{n} u_i(\lambda)\to\min_{\lambda\in\mathbb R^d},
\end{align}
where $v_i$ and $u_i$ are Fenchel convex and concave conjugates of $g_i$ and $f_i$ respectively \cite[Chap.\,2]{Barbu2012}:
\[ v_i(\lambda)=\sup_{y_i\in Y_i}(\langle\lambda,y_i)-g_i(y_i)),\quad u_i(\lambda)=\inf_{x_i\in X_i}(\langle\lambda,x_i\rangle - f_i(x_i)).\]
The dual problem (\ref{1.5}) is solvable and there is no duality gap: $F(x^*,y^*)=G(\lambda^*)$ \cite[Theorem A.1]{Beck2017}.

Is is easy to see that the function $-F$ is strongly convex on $Z$ with parameter 
\begin{align} \label{1.6}
 \sigma=\min\left\{\min_{1\le i\le m}\sigma_i', \min_{1\le i\le n}\sigma_i''\right\}. 
\end{align} 
The Lagrange function (\ref{1.3}) is $\sigma$-strongly concave in $(x,y)$. Hence, for
\[ \widetilde z(\lambda)=\arg\max_{z\in Z} L(z,\lambda)\]
we have (see \cite[Theorem 5.25]{Beck2017})
\[ L(\widetilde z(\lambda),\lambda)-L(z,\lambda)\ge\frac{\sigma}{2}\|z-\widetilde z(\lambda)\|^2.  \]
Since $G(\lambda)=L(\widetilde z(\lambda),\lambda)$, $G(\lambda^*)=F(z^*)=L(z^*,\lambda)$, by putting $z=z^*$, we get
\begin{align} \label{1.7}
 G(\lambda)-G(\lambda^*)\ge\frac{\sigma}{2}\|z^*-\widetilde z(\lambda)\|^2.
\end{align} 

Furthermore, $F$ is $K$-Lipshitz:
\begin{align}
|F(z)-F(z')| & \le K\|z-z'\|,\quad z,z'\in Z, \label{1.8}\\
K &=\left(\sum_{i=1}^m (K_i')^2+\sum_{i=1}^n (K_i'')^2\right)^{1/2}. \label{1.9}
\end{align}
 From (\ref{1.8}), (\ref{1.7}) it follows that
\begin{align} \label{1.10}
 |F(\widetilde z(\lambda))-F(z^*)|\le K\|\widetilde z(\lambda))-z^*\|\le K\sqrt{\frac{2}{\sigma} (G(\lambda)-G(\lambda^*))}.
\end{align}

Recall that a function $f:A\mapsto\mathbb R$ is called $\kappa$-smooth if
\[ \| \nabla f(x)-\nabla f(y)\|\le\kappa\|x-y\|,\quad x,y\in A.\]
From Assumption 2, by Theorem 5.26 of \cite{Beck2017}, it follows that the functions $u_i$ are $1/\sigma_i'$-smooth, and $v_i$ are $1/\sigma_i''$-smooth. It easily follows that $G$ is also smooth with the smoothness parameter
\begin{align} \label{1.11}
 \kappa=\sum_{i=1}^m\frac{1}{\sigma_i'}+\sum_{i=1}^n\frac{1}{\sigma_i''}.
\end{align} 
Therefore (see \cite[Theorem 5.8]{Beck2017}),
\begin{align} \label{1.12}
G(\lambda)-G(\lambda')\ge\langle\nabla G(\lambda'),\lambda-\lambda'\rangle+\frac{1}{2\kappa}\|\nabla G(\lambda)-\nabla G(\lambda') \|^2.
\end{align}
Furthermore, since 
\begin{align} \label{1.13}
 \nabla v_i(\lambda)=\widetilde y_i(\lambda),\quad \nabla u_i(\lambda)=\widetilde x_i(\lambda),
 \end{align}
(see \cite[Corollary 4.21]{Beck2017}), we have
\[ \nabla G(\lambda)=\Delta \widetilde z(\lambda):=\sum_{i=1}^n\widetilde y_i(\lambda)-\sum_{i=1}^m\widetilde x_i(\lambda). \]
Putting $\lambda'=\lambda^*$, from (\ref{1.12}) we get
\begin{align} \label{1.14}
 \|\Delta\widetilde z(\lambda)\|\le\sqrt{2\kappa(G(\lambda)-G(\lambda^*))},
\end{align} 
since $\nabla G(\lambda^*)=0$, $\Delta\widetilde z(\lambda^*)=0$. 

The inequalities (\ref{1.10}), (\ref{1.14}) show that the difference $G(\lambda)-G(\lambda^*)$ controls the optimality gap $F(z^*)-F(\widetilde z(\lambda))$ as well as the feasibility residual $\Delta\widetilde z(\lambda)$ of the plan $\widetilde z(\lambda)$, stimulated by $\lambda$. So, to get an approximately optimal performance, it is enough to approximately solve the dual problem. 

Let us consider the gradient descent algorithm with some step sizes $\eta_t>0$:
\begin{align} \label{1.15}
 \lambda_{t+1}=\lambda_t-\eta_t\nabla G(\lambda_t)=\lambda_t-\eta_t\Delta \widetilde z(\lambda_t).
\end{align} 
Observing the agent reactions to the transfer price vector $\lambda_t$, the firm manager can iteratively update $\lambda_t$ according to the law of supply and demand: if $\Delta \widetilde z_j(\lambda_t)>0$, then the supply in $j$-th commodity is greater than the demand and the price $\lambda_{t,j}$ decreases, and if $\Delta \widetilde z_j(\lambda_t)<0$, the price $\lambda_{t,j}$ goes up.

The are many results concerning the rate of convergence of the gradient descent method, but they require some additional knowledge concerning $G$, and hence the revenue and cost functions of the agents. To get fast convergence rate let us apply Nesterov's fast gradient descent algorithm \cite{Nesterov1983} in the form considered in \cite{Su2016}:
\begin{align}
\lambda_t &=\mu_{t-1}-\eta\nabla G(\mu_{t-1})=\mu_{t-1}-\eta\Delta \widetilde z(\mu_{t-1}), \label{1.16}\\
\mu_t &=\lambda_t+\frac{t-1}{t+2}(\lambda_t-\lambda_{t-1}),\quad \mu_0=\lambda_0\in\mathbb R^d. \label{1.17}
\end{align}
Note that this algorithm comes from a slightly general family than (\ref{1.15}). If $\eta\le 1/\kappa$ then (see \cite{Su2016})
\begin{align} \label{1.18}
 G(\lambda_t)-G(\lambda^*)\le \frac{2\|\lambda_0-\lambda^*\|^2}{\eta(t+1)^2}. 
\end{align}

From (\ref{1.10}), (\ref{1.14}), (\ref{1.18}) we get the following result.
\begin{theorem} \label{th:2}
If $\eta\in (0,1/\kappa]$, then for the transfer price sequence $\lambda_t$, generated by the fast gradient descent algorithm (\ref{1.16}), (\ref{1.17}), we have
\begin{align*}
|F(\widetilde z(\lambda_t))-F(z^*)| &\le \frac{2K}{\sqrt{\sigma\eta}}\frac{\|\lambda_0-\lambda^*\|}{t+1},\\
\|\Delta\widetilde z(\lambda_t)\| &\le 2\sqrt\frac{\kappa}{\eta}\frac{\|\lambda_0-\lambda^*\|}{t+1}.
\end{align*}
where the constants $\sigma$, $K$, $\kappa$ are defined by (\ref{1.6}), (\ref{1.9}), (\ref{1.11}).
\end{theorem}

To implement the algorithm (\ref{1.16}), (\ref{1.17}) one needs to know the smoothness parameter $\kappa$ of $G$. However, in practice it is unknown. In what follows we show that it is possible to construct a convergent algorithm without any parameter knowledge by sacrificing a large amount of convergence rate. To this end consider the SOLO FTRL (Scale-free Online Linear Optimization Follow The  Regularized Leader) algorithm of \cite{Orabona2018}: 
\begin{align}
 \lambda_0&=0,\quad L_0=0, \nonumber\\
\lambda_t &=\arg\min_{\lambda\in\mathbb R^d}\left( \langle L_{t-1},\lambda\rangle+\sqrt{\sum_{j=1}^{t-1} \|\nabla G(\lambda_j)\|^2}\cdot\|\lambda\|^2/2\right), \quad t\ge 1, \label{1.21}\\
L_t &=L_{t-1}+\nabla G(\lambda_t). \nonumber
\end{align}
Solving the optimization problem (\ref{1.21}), we get
\begin{align} \label{1.22}
\lambda_t=-\frac{L_{t-1}}{\sqrt{\sum_{j=1}^{t-1} \|\nabla G(\lambda_j)\|^2}}=-\frac{\sum_{j=1}^{t-1}\Delta\widetilde z(\lambda_j)}{\sqrt{\sum_{j=1}^{t-1} \|\Delta\widetilde z(\lambda_j)\|^2}},\quad  \lambda_0=0.
\end{align}
Note that if $\|\Delta\widetilde z(\lambda_0)\|^2=0$, then $\lambda_0$ is an optimal solution, which is not the case for our problem.

In the online learning theory the quantity 
\[\mathscr R_T(\lambda)=\sum_{t=1}^T(G(\lambda_t)-G(\lambda))\]
is called the regret of an algorithm, generating $\lambda_t$, with respect to a fixed decision $\lambda$. 
Using the convexity of $G$, we can bound the regret as follows:
\[ \mathscr R_T(\lambda)\le\sum_{t=1}^T \langle\nabla G(\lambda_t),\lambda_t\rangle-\sum_{t=1}^T\langle\nabla G(\lambda_t),\lambda\rangle. \]
This simple and well-known linearization argument reduces any online convex optimization problem to the online linear optimization problem. According to \cite[Theorem 1]{Orabona2018} we have the following bound for the regret of the algorithm (\ref{1.22}):
\begin{equation} \label{1.23}
 \mathscr R_T(\lambda)\le\left(\|\lambda\|^2/2+2.75 \right)\sqrt{\sum_{t=1}^T \|\nabla G(\lambda_t)\|^2}+3.5\sqrt{T-1} \max_{t\le T}\|\nabla G(\lambda_t) \|. 
\end{equation} 

Let us estimate $\nabla G(\lambda)$:
\begin{align}
\|\nabla G(\lambda)\| &\le\sum_{i=1}^n\|\nabla v_i(\lambda)\|+\sum_{i=1}^m\|\nabla u_i(\lambda)\|=\sum_{i=1}^n\|\widetilde y_i(\lambda)\|+\sum_{i=1}^m\|\widetilde x_i(\lambda)\|\| \nonumber\\
& \le  (m+n)c\sqrt d. \label{1.24}
\end{align}
Here we used (\ref{1.13}) and Assumption 1, which allows to estimate the diameter of any set $X_i$, $Y_i$ by $c\sqrt d$. For 
\begin{align} \label{1.25}
\overline\lambda_T=\frac{1}{T}\sum_{t=1}^T \lambda_t
\end{align}
using the convexity of $G$ and the inequalities (\ref{1.23}), (\ref{1.24}), we get
\begin{align} \label{1.26}
G(\overline\lambda_T)-G(\lambda^*)&\le\frac{1}{T}\sum_{t=1}^T(G(\lambda_t)-G(\lambda^*))=\frac{\mathscr R_T(\lambda^*)}{T}\nonumber\\
&\le\frac{(m+n)c\sqrt d}{\sqrt T}\left(\frac{\|\lambda^*\|^2}{2}+6.25 \right).
\end{align}

From the last inequality and (\ref{1.10}), (\ref{1.14}) we obtain the following result.
\begin{theorem} \label{th:3}
For the average transfer price vector (\ref{1.25}), generated by the SOLO FTRL algorithm (\ref{1.22}), we have
\begin{align*}
  F(z^*)-F(\widetilde z(\overline\lambda_T)) & \le K\sqrt{\frac{(m+n)c}{\sigma}}\sqrt{\|\lambda^*\|^2+12.5}\frac{d^{1/4}}{T^{1/4}},\\
 \|\Delta\widetilde z(\overline\lambda_T)\| &\le \sqrt{\kappa(m+n)c}\sqrt{\|\lambda^*\|^2+12.5}\frac{d^{1/4}}{T^{1/4}}.
\end{align*}
\end{theorem}

So, for optimality and feasibility residuals we have the estimates of order $T^{-1/4}$ instead of the estimates of order $T^{-1}$ of the fast gradient descent algorithm. However the SOLO FTRL algorithm requires no knowledge about the revenue and cost functions, and contains no parameters.

In addition, the following lemma show that the iterations $\lambda_t$ are uniformly bounded. This result will be useful in the next section.
\begin{lemma} \label{lem:1}
The iterations (\ref{1.21}) satisfy the inequalities
\[ -1\le \lambda_{t,k}\le K'+1,\quad K'=\max_{i=1,\dots,m}K_i'.   \]
\end{lemma}
\begin{proof}
The inequalities are true for $\lambda_0=0$. Assume that they are satisfied for $\lambda_t$. 

(i) Lower bound. Let $\lambda_{t,k}\in [-1,0]$. Since the functions $f_i$, $g_i$ are non-decreasing in each argument, from the definitions of $\widetilde x_i(\lambda)$, $\widetilde y_i(\lambda)$ it follows that 
\[ \widetilde x_{i,k}(\lambda_t)=c,\quad \widetilde y_{i,k}(\lambda_t)=0.\]
Note that $\Delta\widetilde z_k(\lambda_t)=\sum_{i=1}^n\widetilde y_{i,k}(\lambda)-\sum_{i=1}^m\widetilde x_{i,k}(\lambda)=-mc$ and
\begin{equation} \label{1.28}
\lambda_{t+1,k} =\lambda_{t,k}\sqrt{\frac{\sum_{j=1}^{t-1} \|\Delta \widetilde z(\lambda_j)\|^2}{\sum_{j=1}^t \|\Delta \widetilde z(\lambda_j)\|^2}}-\frac{\Delta\widetilde z_k(\lambda_t)}{\sqrt{\sum_{j=1}^t \|\Delta \widetilde z(\lambda_j)\|^2}}.
\end{equation}
Hence, $\lambda_{t+1,k} \ge  \lambda_{t,k}\ge -1.$ 

If $\lambda_{t,k}\ge 0$, then from (\ref{1.28}) we get
\[ \lambda_{t+1,k}\ge -\frac{\Delta\widetilde z_k(\lambda_t)}{\sqrt{\sum_{j=1}^t \|\Delta \widetilde z(\lambda_j)\|^2}}\ge - \frac{|\Delta\widetilde z_k(\lambda_t)|}{\sqrt{\sum_{j=1}^t \|\Delta \widetilde z(\lambda_j)\|^2}}\ge -1. \]

(ii) Upper bound. Let $\lambda_{t,k}\in (K',K'+1]$. Without loss of generality assume that $k=1$. If $\widetilde x_{t,1}(\lambda_t)>0$, then
\[ f_i(\widetilde x_{i,1}(\lambda_t),\widetilde x_{i,2}(\lambda_t),\dots,\widetilde x_{i,m}(\lambda_t))-\lambda_{t,1}\widetilde x_{i,1}(\lambda_t)\ge f_i(0,\widetilde x_{i,2}(\lambda_t),\dots,\widetilde x_{i,m}(\lambda_t)),\]
and we get a contradiction: $\lambda_{t,1}\le K'$, since
\begin{align*}
\lambda_{t,1}\widetilde x_{i,1}(\lambda_t) &\le f_i(\widetilde x_{i,1}(\lambda_t),\widetilde x_{i,2}(\lambda_t),\dots,\widetilde x_{i,m}(\lambda_t))-f_i(0,\widetilde x_{i,2}(\lambda_t),\dots,\widetilde x_{i,m}(\lambda_t))\\
&\le K_i'|\widetilde x_{i,1}(\lambda_t)| \le K' \widetilde x_{i,1}(\lambda_t).
\end{align*}
Thus, $\widetilde x_{t,1}(\lambda_t)=0$ and $\Delta\widetilde z_1(\lambda_t)\ge 0$. Now from (\ref{1.28}) it follows that $\lambda_{t+1,1}\le\lambda_{t,1}\le K'$.

If $\lambda_{t,k}\in [0,K']$, then again from (\ref{1.28}) we get
\[ \lambda_{t+1,k}\le\lambda_{t,k}+ \frac{|\Delta\widetilde z_k(\lambda_t)|}{\sqrt{\sum_{j=1}^t \|\Delta \widetilde z(\lambda_j)\|^2}}\le K'+1.\quad \square\]
\end{proof}

\section{Dynamic problem} \label{sec:3}
Consider a sequence of time dependent revenue and cost functions
\[ f_{t,i}:X_i\mapsto\mathbb R_+,\quad g_{t,i}:Y_i\mapsto\mathbb R_+,  \]
and the sequence of profit maximization problems
\begin{align*}
 F_t(x,y)=\sum_{i=1}^m f_{t,i}(x_i)-\sum_{i=1}^n g_{t,i}(y_i)\to\max_{(x,y)\in S}, 
\end{align*}  
where $S$ is defined by (\ref{1.2}). For notational simplicity we assume that $X_i$, $Y_i$ do not depend on $t$. Furthermore, we assume that the Assumptions 1\,--\,3 are still satisfied for $f_{t,i}$, $g_{t,i}$ instead of $f_i$, $g_i$, wherein all constants $\varepsilon$, $c$, $\sigma_i'$, $\sigma_i''$, $K_i'$, $K_i''$ are independent of $t$.

As in the static case, put $\widetilde z_t(\lambda)=(\widetilde x_t(\lambda),\widetilde y_t(\lambda))$,
\begin{align*}
\widetilde x_{t,i}(\lambda)&=\arg\max_{x_i\in X_i}(f_{t,i}(x_i)-\langle\lambda,x_i\rangle),\quad i=1,\dots,m,\\
\widetilde y_{t,i}(\lambda)&=\arg\max_{y_i\in Y_i}(\langle\lambda,y_i\rangle-g_{t,i}(y_i)),\quad i=1,\dots,n. 
\end{align*}
Applying the SOLO FTRL algorithm (\ref{1.22}) to $G_t$ instead of $G$,  by \cite[Theorem 1]{Orabona2018} we get the estimate similar to (\ref{1.26}):
\begin{align} \label{2.4A}
\mathscr R_T(\lambda)=\sum_{t=1}^T(G_t(\lambda_t)-G_t(\lambda))\le (m+n)c\sqrt d\left(\frac{\|\lambda\|^2}{2}+6.25 \right)\sqrt T,
\end{align} 
since $\nabla G_t(\lambda)$ still satisfies (\ref{1.24})
\begin{align}
\|\nabla G_t(\lambda)\| \le  (m+n)c\sqrt d. \label{2.4}
\end{align}

We want to estimate the regret with respect to the best possible plan sequence $z_t^*$:
\[  \sum_{t=1}^T(F_t(z_t^*)-F_t(\widetilde z_t(\lambda_t)). \]
Note, that in (\ref{2.4A}) $\lambda$ is fixed, but optimal solutions $\lambda_t^*$ of the dual problems
\begin{align} \label{2.2}
G_t(\lambda)=\sum_{i=1}^m v_{t,i}(\lambda)-\sum_{i=1}^{n} u_{t,i}(\lambda)\to\min_{\lambda\in\mathbb R^d},
\end{align}
\[ v_{t,i}(\lambda)=\sup_{y_i\in Y_i}(\langle\lambda,y_i)-g_{t,i}(y_i)),\quad u_{t,i}(\lambda)=\inf_{x_i\in X_i}(\langle\lambda,x_i\rangle - f_{t,i}(x_i))\]
depend on $t$. Thus, we cannot apply (\ref{1.10}), (\ref{1.14}) to pass to the regret, related to the residuals in the primal problem, as in the proof of Theorem \ref{th:3}. 

Recall that the notation $X_t=O_{\mathsf P}(c_t)$, where $X_t$ are random variables, $\mathsf P$ is a probability measure, and $c_t>0$ are constants, means that  the sequence $X_t/c_t$ is stochastically bounded. That is, for any $\varepsilon>0$ there exists $C>0$, $t_0>0$ such that
\[ \mathsf P(|X_t|/c_t\ge C)\le\varepsilon,\quad t\ge t_0.\]

It is easy to see that Lemma \ref{lem:1} is still holds true for the iterations 
\begin{align} \label{2.3}
\lambda_t=-\frac{\sum_{j=1}^{t-1}\Delta\widetilde z_t(\lambda_j)}{\sqrt{\sum_{j=1}^{t-1} \|\Delta\widetilde z_j(\lambda_j)\|^2}},\quad  \lambda_0=0,
\end{align}
constructed for time-dependent functions $f_{i,t}$, $g_{i,t}$. By this lemma the sequence $\lambda_t$ is uniformly bounded:
\begin{align} \label{2.5}
 \|\lambda_t\|\le b\sqrt{d},\quad b:=K'+1.
\end{align} 

\begin{theorem} \label{th:4}
Assume additionally that the revenue and cost functions are of the form
\[ f_{t,i}(x_i)=f_i(x_i,\xi_{t,i}),\quad g_{t,i}(y_i)=g_i(y_i,\eta_{t,i}),\]
where $(\xi_t,\eta_t)\in\Theta\subset\mathbb R^m\times\mathbb R^n$ is a sequence of i.i.d. random vectors such that 
\begin{align} \label{2.10A}
\mathsf E \sup_{x_i\in X_i} f_i(x_i,\xi_{t,i})<\infty, \quad \mathsf E\sup_{y_i\in Y_i} g_i(y_i,\eta_{t,i})<\infty.
\end{align}
Then the price sequence (\ref{2.3}) generated by the SOLO FTRL algorithm, ensures no-regret learning with respect to the best possible plan sequence $z_t^*$:
\begin{align} \label{2.11}
 \frac{1}{T}\sum_{t=1}^T(F_t(z_t^*)-F_t(\widetilde z_t(\lambda_t))\to 0\quad \mathrm{a.s},\quad T\to\infty, 
\end{align}
and the estimate 
\begin{align} \label{2.12}
 \frac{1}{T}\sum_{t=1}^T(F_t(z_t^*)-F_t(\widetilde z_t(\lambda_t))=O_{\mathsf P}\left(\frac{1}{T^{1/4}}\right).
 \end{align}
 The  equilibrium between the supply and demand is satisfied on average:
\begin{equation} \label{2.12A}
\frac{1}{T}\sum_{t=1}^T\Delta \widetilde z_t(\lambda_t)\to 0\quad \mathrm{a.s.},\quad T\to\infty,
\end{equation}
\begin{equation} \label{2.12B}
\frac{1}{T}\sum_{t=1}^T\Delta \widetilde z_t(\lambda_t)=O_{\mathsf P}\left(\frac{1}{T^{1/4}} \right).
\end{equation}
\end{theorem}
\begin{proof} (1) From the definition (\ref{2.2}) of the dual objective functions we get
\[G_t(\lambda)=F_t(\widetilde z_t(\lambda))-\langle \lambda,\Delta\widetilde z_t(\lambda)\rangle,\quad \Delta\widetilde z_t(\lambda):=\sum_{i=1}^m\widetilde y_{t,i}(\lambda)-\sum_{i=1}^n\widetilde x_{t,i}(\lambda).\] 
Since $\Delta\widetilde z_t(\lambda)=\nabla G_t(\lambda)$, and $\lambda_t^*$ is a minimum point of $G_t$, using the strong duality we obtain the inequality
\begin{align*}
&F_t(z^*_t)-F_t(\widetilde z_t(\lambda_t))=G_t(\lambda^*_t)-G_t(\lambda_t)-\langle\lambda_t,\nabla G_t(\lambda_t)\rangle \nonumber \\
&\le -\langle\lambda_t,\nabla G_t(\lambda_t)\rangle=-\langle\lambda_t,\nabla G_t(\lambda)\rangle+\langle\lambda_t,\nabla G_t(\lambda)-\nabla G_t(\lambda_t)\rangle \nonumber\\
&\le -\langle\lambda_t,\nabla G_t(\lambda)\rangle+\|\lambda_t\|\cdot\|\nabla G_t(\lambda)-\nabla G_t(\lambda_t)\| 
\end{align*}
for any $\lambda$. Take the average:
\begin{align} 
&\frac{1}{T}\sum_{t=1}^T (F_t(z^*_t)-F_t(\widetilde z_t(\lambda_t)))\le -\frac{1}{T}\sum_{t=1}^T
\langle\lambda_t,\nabla G_t(\overline\lambda)\rangle+\frac{1}{T}\sum_{t=1}^T\|\lambda_t\|\cdot\|\nabla G_t(\overline\lambda)-\nabla G_t(\lambda_t)\| \nonumber\\
&\le -\frac{1}{T}\sum_{t=1}^T
\langle\lambda_t,\nabla G_t(\overline\lambda)\rangle+\sqrt{\frac{1}{T}\sum_{t=1}^T\|\lambda_t\|^2}\sqrt{\frac{1}{T}\sum_{t=1}^T\|\nabla G_t(\overline\lambda)-\nabla G_t(\lambda_t)\|^2}. \label{2.18}
\end{align}

(2) Let us prove the existence of $\overline\lambda$, satisfying the equation
\begin{align} \label{2.14}
\mathsf E\mathsf \nabla G_t(\overline\lambda)=0.
\end{align}
From (\ref{2.10A}) it follows that $\mathsf E G_t(\lambda)<\infty$, since $G_t=\sum_{i=1}^m v_{t,i}-\sum_{i=1}^{n} u_{t,i}$,
\[ u_{t,i}(\lambda)\le\sup_{x_i\in X_i}\langle\lambda,x_i\rangle +\sup_{x_i\in X_i} f_{t,i}(x_i),\quad  
v_{t,i}(\lambda)\le\sup_{y_i\in Y_i}\langle\lambda,y_i\rangle+\sup_{y_i\in Y_i} g_{t,i}(y_i).\]
Put
\[ J_+(\lambda)=\{j:\lambda_j>0\},\quad J_{-}(\lambda)=\{j:\lambda_j<0\}, \]
\[ x_\varepsilon(\lambda)=\varepsilon\sum_{j\in J_{-}(\lambda)} e_j,\quad y_\varepsilon(\lambda)=\varepsilon\sum_{j\in J_{+}(\lambda)} e_j,\]
where $(e_1,\dots,e_d)$ is the standard basis of $\mathbb R^d$. Then
\begin{align*}
G_t(\lambda)&=\sum_{i=1}^m\sup_{x_i\in X_i}\{f_{t,i}(x_i)-\langle\lambda,x_i\rangle\}+\sum_{i=1}^n\sup_{y_i\in Y_i}\{\langle\lambda,y_i\rangle-g_{t,i}(y_i)\}\\
&\ge\sum_{i=1}^m(f_{t,i}(x_\varepsilon)-\langle\lambda,x_\varepsilon\rangle)+ \sum_{i=1}^n(\langle\lambda,y_\varepsilon\rangle-g_{t,i}(y_\varepsilon))\\
&\ge -m\varepsilon\sum_{j\in  J_{-}(\lambda)}\lambda_j+n\varepsilon \sum_{j\in J_{+}(\lambda)}\lambda_j-\sum_{i=1}^n g_{t,i}(y_\varepsilon)\\
&\ge\varepsilon\min\{m,n\}\sum_{j=1}^d|\lambda_j|-\sum_{i=1}^n\sup_{y_i\in Y_i} g_{t,i}(y_i).
\end{align*}
It follows that $G_t(\lambda)\to+\infty$ as $\|\lambda\|\to\infty$. Furthermore, since $G_t$ are non-negative:
\[ G_t(\lambda)\ge G_t(\lambda^*_t)\ge F_t(0)=0,\]
we can apply Fatou's lemma to get
\[  \liminf_{t\to\infty}\mathsf E G_t(\lambda)\ge\mathsf E  \liminf_{t\to\infty} G_t(\lambda)=\infty.\]
This implies the existence of a (global) minimum point $\overline\lambda$ of the function $\lambda\mapsto\mathsf E G_t(\lambda)$. Note that we can take $\overline\lambda$ independent of $t$, since $\xi_t,\eta_t$ are identically distributed, and $\mathsf E G_t(\lambda)$ does not depend on $t$.

As $\nabla G_t(\lambda)$ is uniformly  bounded, we can interchange the expectation and differentiation (see \cite[Proposition 2.3]{Asmussen2007}):
$\nabla \mathsf E G_t(\overline\lambda)=\mathsf E\nabla G_t(\overline\lambda)$.
 By the optimality criterion we get (\ref{2.14}):
\[ \nabla \mathsf E G_t(\overline\lambda)=\mathsf E\nabla G_t(\overline\lambda)=0.\]

(3) Assume that $(\xi_t,\eta_t)$ are defined on some probability space $(\Omega,\mathscr F,\mathsf P)$. Let $\mathscr F_t$ be the $\sigma$-algebra generated by $((\xi_j,\eta_j))_{j=1}^t$, $t\ge 1$, and put $\mathscr F_{-1}=\mathscr F_0=\{\emptyset,\Omega\}$. Then $U_t=\langle\lambda_t,\nabla G_t(\overline\lambda)\rangle$ forms the martingale difference sequence with respect to the filtration $\mathscr F_t$:
\[ \mathsf E\left(\langle\lambda_t,\nabla G_t(\overline\lambda)\rangle|\mathscr F_{t-1}\right)=\langle \lambda_t,\mathsf E(\mathsf \nabla G_t(\overline\lambda) |\mathscr F_{t-1})\rangle=\langle \lambda_t,\mathsf E\mathsf \nabla G_t(\overline\lambda)\rangle= 0.\]
We used the fact that $\lambda_t$ is $\mathscr F_{t-1}$-measurable, and $G_t$ is independent from $\mathscr F_{t-1}$. Since, by (\ref{2.4}), (\ref{2.5}), 
\[ \sum_{t=1}^\infty\frac{\mathsf E U_t^2}{t^2}\le
\sum_{t=1}^\infty\frac{\mathsf E \left(\|\lambda_t\|^2 \|\nabla G_t\|^2\right)}{t^2}\le (c b d (m+n))^2\sum_{t=1}^\infty\frac{1}{t^2}<\infty,   \]
the martingale $\sum_{j=1}^t U_j/j$ is bounded in $L^2$. It follows that $\sum_{j=1}^t U_j/j$ converges a.s. \cite[Chap.\,12]{Williams1991}. From the Kronecker lemma \cite[Chap IV, \S3]{Shiryaev1996} we get the strong law of large numbers:
\begin{align} \label{2.15}
 \frac{1}{T}\sum_{t=1}^T U_t=\frac{1}{T}\sum_{t=1}^T\langle\lambda_t,\nabla G_t(\overline\lambda)\rangle\to 0\quad \mathrm{a.s.} 
\end{align}

Consider the inequality (\ref{1.12}): 
\begin{align} \label{2.16}
 G_t(\lambda_t)-G_t(\overline\lambda)\ge\langle\nabla G_t(\overline\lambda),\lambda_t-\overline\lambda\rangle+\frac{1}{2\kappa}\|\nabla G_t(\lambda_t)-\nabla G_t(\overline\lambda)\|^2.
 \end{align}
Applying the relations (\ref{2.4A}) and (\ref{2.15}), noting that the last one also holds true with $\overline\lambda$ instead of $\lambda_t$, we get
\begin{align} \label{2.17}
 \frac{1}{T}\sum_{t=1}^T\|\nabla G_t(\lambda_t)-\nabla G_t(\overline\lambda)\|^2
\to 0\quad\textrm{a.s.}
\end{align}
We see that the assertion (\ref{2.11}) follows from (\ref{2.18}), (\ref{2.15}), (\ref{2.17}), since $\lambda_t$ are uniformly bounded. 

The ``average equilibrium`` property (\ref{2.12A}) also follows immediately. We have
\begin{align*}
\frac{1}{T}\sum_{t=1}^T\Delta \widetilde z_t(\lambda_t)=\frac{1}{T}\sum_{t=1}^T\nabla G_t(\lambda_t)=
\frac{1}{T}\sum_{t=1}^T\nabla G_t(\overline\lambda)+\frac{1}{T}\sum_{t=1}^T (\nabla G_t(\lambda_t)-\nabla G_t(\overline\lambda)).
\end{align*}
The first term conveges to zero a.s. by the mentioned strong law of large numbers. The second term converges  to zero a.s. by  (\ref{2.17}):
\[ \frac{1}{T}\|\sum_{t=1}^T (\nabla G_t(\lambda_t)-\nabla G_t(\overline\lambda))\|\le\sqrt{\frac{1}{T}\sum_{t=1}^T\|\nabla G_t(\lambda_t)-\nabla G_t(\overline\lambda)\|^2}.\]

(4) To prove (\ref{2.12}) we again use (\ref{2.18}). Since $\lambda_t$ are uniformly bounded: $\|\lambda_t\|\le b\sqrt{d}$, for any $A>0$ we have
\begin{align*}
&\mathsf P\left(\frac{1}{T}\sum_{t=1}^T (F_t(z^*_t)-F_t(\widetilde z_t(\lambda_t)))\ge\frac{A}{T^{1/4}}\right)\le \mathsf P\left( -\frac{1}{T}\sum_{t=1}^T
\langle\lambda_t,\nabla G_t(\overline\lambda)\rangle\ge \frac{A}{2 T^{1/4}} \right)\\
&+\mathsf P\left(\sqrt{\frac{1}{T}\sum_{t=1}^T\|\nabla G_t(\overline\lambda)-\nabla G_t(\lambda_t)\|^2} \ge \frac{A}{2b\sqrt{d} T^{1/4}}\right)
\end{align*}
By the inequalities (\ref{2.16}) and (\ref{2.4A}) the condition
\begin{align} \label{2.18A}
\frac{1}{T}\sum_{t=1}^T\|\nabla G_t(\overline\lambda)-\nabla G_t(\lambda_t)\|^2\ge\frac{A^2}{4 b^2 d T^{1/2}}  
\end{align}
implies that
\begin{align}
&-\frac{1}{T}\sum_{t=1}^T\langle\lambda_t,\nabla G_t(\overline\lambda)\rangle\ge \frac{A^2}{8\kappa b^2 d \sqrt T}-\frac{1}{T}\sum_{t=1}^T(G_t(\lambda_t)-G_t(\overline\lambda))\nonumber\\
&\ge \left(\frac{A^2}{8\kappa b^2 d}- c (m+n) \sqrt{d}\left(\|\overline\lambda\|^2/2+6.25\right)\right)\frac{1}{\sqrt T}\ge \frac{A^2}{16\kappa b^2 d} \frac{1}{\sqrt T} \label{2.18B}
\end{align}
for sufficiently large $A$:
\[ \frac{A^2}{16\kappa b^2 d}\ge c (m+n) \sqrt{d}\left(\|\overline\lambda\|^2/2+6.25\right).\]
Hence,
\begin{align}
&\mathsf P\left(\frac{1}{T}\sum_{t=1}^T (F_t(z^*_t)-F_t(\widetilde z_t(\lambda_t)))\ge\frac{A}{T^{1/4}}\right)\le \mathsf P\left( -\frac{1}{T}\sum_{t=1}^T
\langle\lambda_t,\nabla G_t(\overline\lambda)\rangle\ge \frac{A}{2 T^{1/4}} \right)\nonumber\\
&+\mathsf P\left(-\frac{1}{T}\sum_{t=1}^T\langle\lambda_t,\nabla G_t(\overline\lambda)\rangle\ge  \frac{A^2}{16\kappa b^2 d} \frac{1}{\sqrt T}\right). \label{2.19}
\end{align}

Furthermore, in view of the estimate
\[ |\langle\lambda_t,\nabla G_t(\overline\lambda)\rangle|\le\|\lambda_t\|\cdot\|\nabla G_t(\overline\lambda)\|\le M:=cb (m+n) d\]
we can apply by the Azuma-Hoeffding inequality (see \cite[Proposition 10.5.1]{Lange2010}):
\begin{align} \label{2.21}
 \mathsf P\left(-\frac{1}{T}\sum_{t=1}^T \langle\lambda_t,\nabla G_t(\overline\lambda)\rangle\ge B\right) \le\exp\left(-\frac{B^2 T^2}{2\sum_{t=1}^T M^2}\right)=\exp\left(-\frac{1}{2}\frac{B^2}{M^2} T\right).
 \end{align}
Applying this inequality to each term in the right-hand side of (\ref{2.19}), we get
\begin{align*}
\mathsf P\left(\frac{1}{T}\sum_{t=1}^T (F_t(z^*_t)-F_t(\widetilde z_t(\lambda_t)))\ge\frac{A}{T^{1/4}}\right) &\le \exp\left(-\frac{A^2}{8 M^2}\sqrt{T}\right)\\
&+\exp\left(-\frac{1}{512}\frac{A^4}{M^2\kappa^2 b^4 d^2}\right)
\end{align*}
for sufficiently large $A$. This implies (\ref{2.12}). 

(5) To prove (\ref{2.12B}) consider the representation
\[ \frac{1}{T}\sum_{t=1}^T\Delta\widetilde z_t(\lambda_t)= \frac{1}{T}\sum_{t=1}^T\nabla G_t(\lambda_t)=\frac{1}{T}\sum_{t=1}^T\nabla G_t(\overline\lambda)+\frac{1}{T}\sum_{t=1}^T(\nabla G_t(\lambda_t)-\nabla G_t(\overline\lambda)).\]
Since $G_t(\overline\lambda)$ are i.i.d. and $\mathsf E\nabla G_t(\overline\lambda)=0$, from the central limit theorem it follows that 
\[\frac{1}{T}\sum_{t=1}^T\nabla G_t(\overline\lambda)=O_\mathsf P\left(\frac{1}{\sqrt T}\right). \]
To get the estimate
\begin{align} \label{2.22}
\frac{1}{T}\sum_{t=1}^T(\nabla G_t(\lambda_t)-\nabla G_t(\overline\lambda))= O_\mathsf P\left(\frac{1}{T^{1/4}}\right)
\end{align}
we note that the inequatilty
\[ \left\| \frac{1}{T}\sum_{t=1}^T(\nabla G_t(\lambda_t)-\nabla G_t(\overline\lambda))\right\|\ge \frac{C}{T^{1/4}} \]
implies the inequality similar to (\ref{2.18A}):
\[  \sum_{t=1}^T\|\nabla G_t(\lambda_t)-\nabla G_t(\overline\lambda)\|^2\ge C^2 T^{1/2}.\]
Applying the inequalities similar to (\ref{2.18B}), (\ref{2.21}), we get (\ref{2.22}). The proof is complete.
$\square$  
\end{proof}

\section{Computer experiments} \label{sec:4}

We performed two computer experiments. In the first one we consider only one commodity, and the revenues and costs $f_i$, $g_i$ are shifted power functions. In the second experiment we consider two commodities, and $f_i$, $g_i$ are quadratic. The parameters of the revenue and cost functions are selected randomly. In the static problems these parameters are sampled only once and are fixed at all iterations. In the dynamic problem at each iteration they are independently sampled from the same distribution.

In the static case the results show fast stabilization of the transfer prices and the difference between the supply and demand. In the dynamic case the same quantities fluctuate around equilibrium after a short transition phase. The code is available at \url{https://github.com/drokhlin/Transfer_prices/}.

\begin{example}
For a single commodity ($d=1$) consider $m=15$ sales divisions and $n=25$ production divisions with the revenue and cost functions
\[ f_i(x_i)=\frac{A_i}{\alpha_i}(x_i+\varepsilon_{1,i})^{\alpha_i}-\frac{A_i}{\alpha_i}\varepsilon_{1,i}^{\alpha_i},\quad
 g_i(y_i)=\frac{B_i}{\beta}(y_i+\varepsilon_{2,i})^{\beta_i}-\frac{B_i}{\beta_i}\varepsilon_{2,i}^{\beta_i}
\]
defined on $X_i=Y_i=[0,c]$, $c=10$. Parameters of these functions were generated as follows:
\[A_i\sim U(0,15),\quad 1-\alpha_i\sim U(0,1),\quad \varepsilon_{1,i}-0.1\sim U(0,1),\]
\[B_i \sim U(0,10),\quad \beta_i-1\sim U(0,3),\quad  \varepsilon_{2,i}-0.1\sim U(0,1),\]
where $U$ is the uniform distribution. The problems (\ref{1.2A}), (\ref{1.2B}) can be solved explicitly:
\[ \widetilde x_i(\lambda)=\begin{cases}
0, & \lambda\ge A_i/\varepsilon_{1,i}^{1-\alpha_i}\\
(A_i/\lambda)^{1/(1-\alpha_i)}-\varepsilon_{1,i}, & \lambda\in \left[A_i/(c+\varepsilon_{1,i})^{1-\alpha_i}, A_i/(\varepsilon_{1,i})^{1-\alpha_i}\right]\\
c, & \lambda\le A_i/(c+\varepsilon_{1,i})^{1-\alpha_i}
\end{cases},
\]
\[ \widetilde y_i(\lambda)=\begin{cases}
0, & \lambda\le B_i\varepsilon_{2,i}^{\beta_i-1}\\
(\lambda/B_i)^{1/(\beta_i-1)}-\varepsilon_{2,i}, & \lambda\in \left[B_i\varepsilon_{2,i}^{\beta_i-1},B_i(c+\varepsilon_{2,i})^{\beta_i-1}\right]\\
c, & \lambda\ge B_i/(c+\varepsilon_{2,i})^{\beta_i-1}
\end{cases}.
\]

In Fig.\,\ref{fig:1} we show the graphs of the transfer price and the difference between the supply and demand in the static problem for one realization of parameters (upper panel), and one realization of the same quantities for the dynamic problem, where the model parameters are sampled at each round (lower panel). 

\begin{figure}[h!]
\centering
\includegraphics[width=\textwidth]{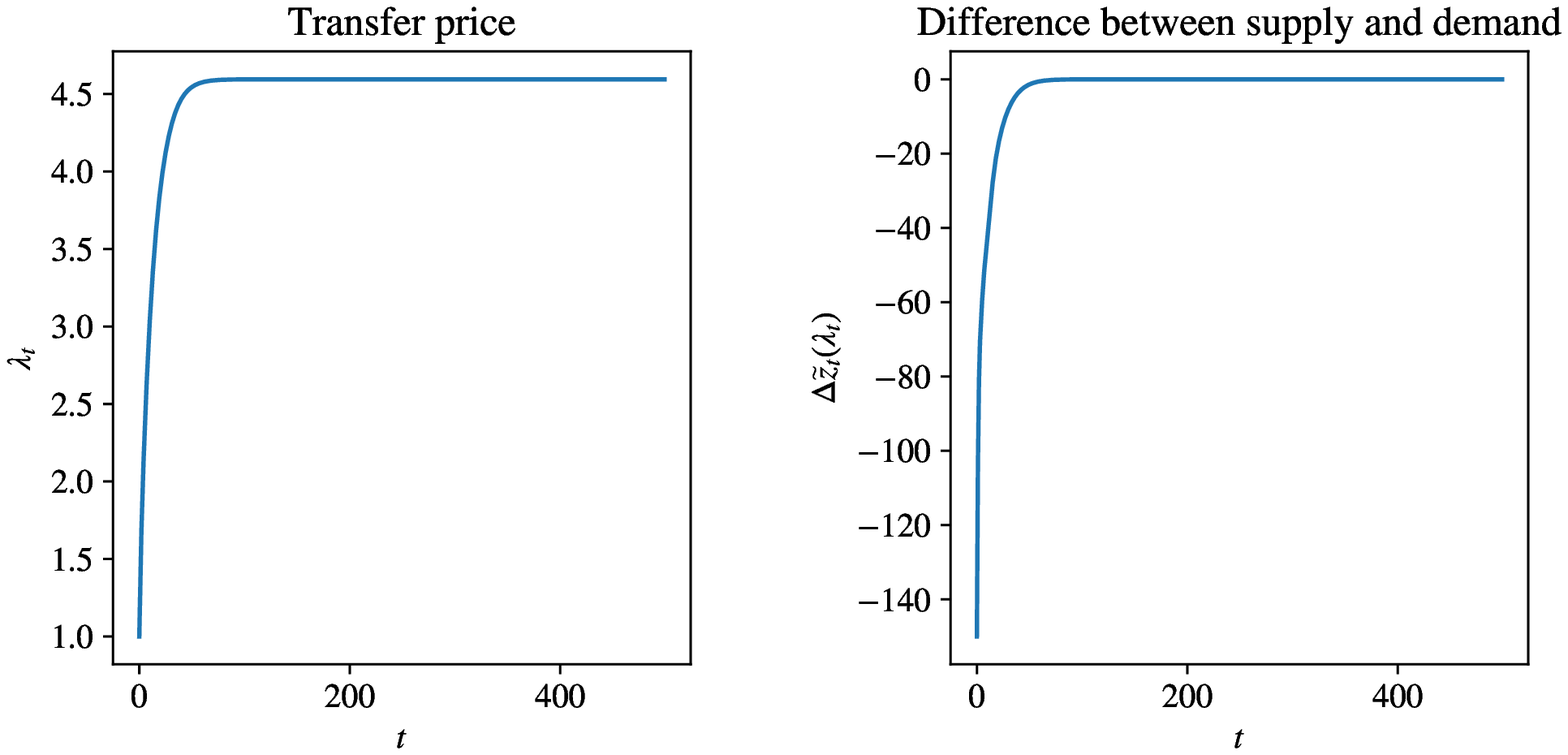}
\includegraphics[width=\textwidth]{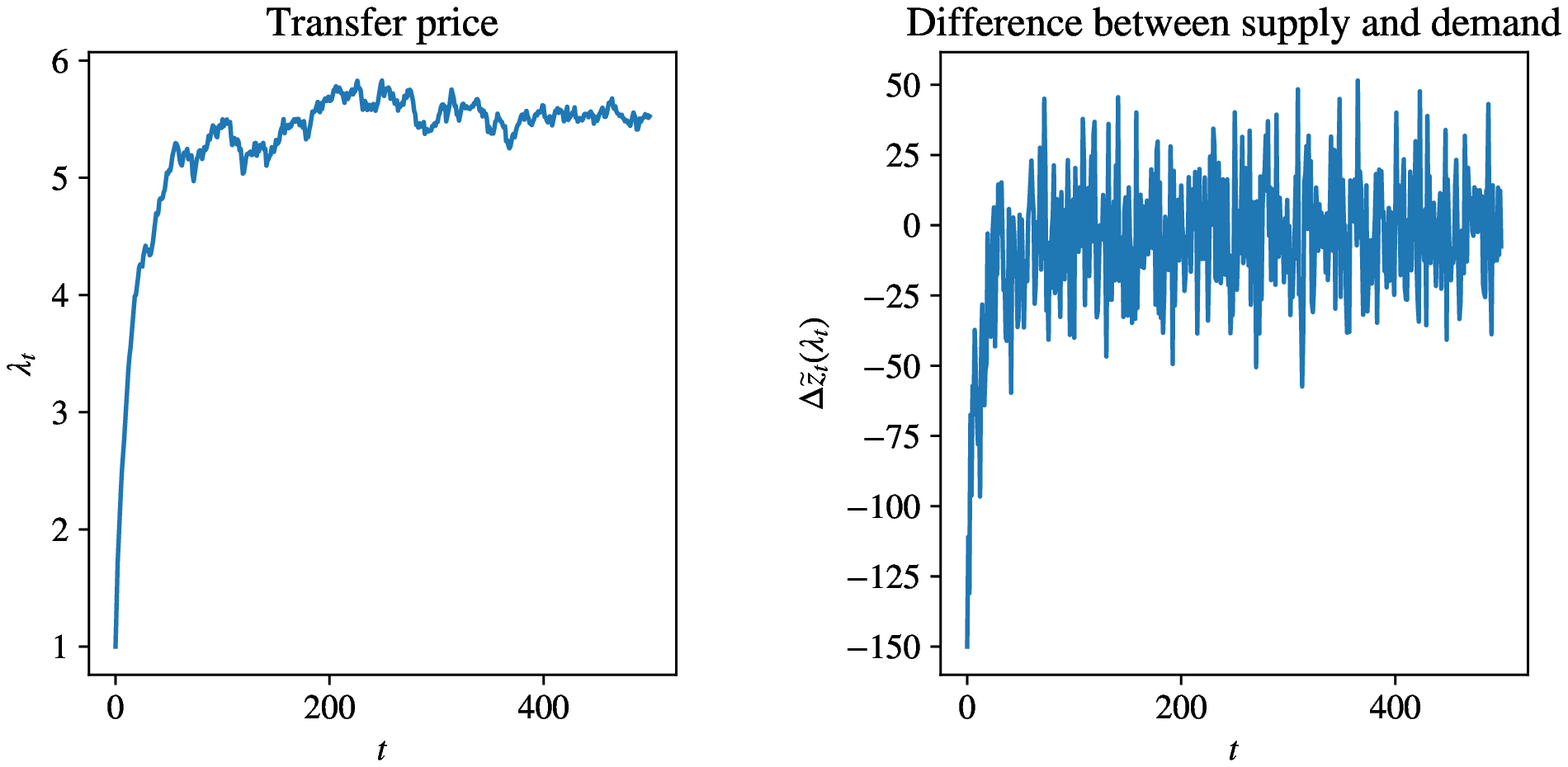}
\caption{One commodity ($d=1$). The transfer price, generated by the SOLO-FTRL algorithm, and the corresponding difference between the supply and demand in the static (upper panel) and dynamic (lower panel) cases.}
 \label{fig:1}
\end{figure}
\end{example}

\begin{example} For two commodities ($d=2$) consider quadratic revenue and cost functions: 
\[ f_i(x_i)=\langle a_i,x_i\rangle-\frac{1}{2}\langle A_i x_i,x_i\rangle,\quad g_i(y_i)=\langle b_i,y_i\rangle+\frac{1}{2}\langle B_i y_i,y_i\rangle,\]
$X_i=Y_i=[0,c]^d$, where symmetric matrices $A_i$, $B_i$ are positive semidefinite. One can regard these expressions as Taylor's approximations of general concave and convex functions near zero. 

To ensure that $f_i$, $g_i$ are non-decreasing in each argument (see Assumption \ref{as:2}) we require that
\[ a_i-A_i x_i \ge 0,\quad b_i+B_i y_i\ge 0\]
componentwise. These condition are satisfied for 
\[ a_i(k)\ge c\sum_{j:A_i(k,j)>0} A_i(k,j), \quad b_i(k)\ge -c\sum_{j:B_i(k,j)<0} B_i(k,j).\]
Here $a_i(k)$, $A_i(k,j)$, etc., are the components of the corresponding vectors and matrices.
In the  experiment we used
\[ A_i, B_i\sim C_i^T C_i+\delta I_d,\]
where the components of $C_i$ are independent standard normal random variables, and
\[ a_i(k)-c\sum_{j:A_i(k,j)>0} A_i(k,j)\sim U(0,1), \quad  b_i(k)+c\sum_{j:B_i(k,j)<0} B_i(k,j)\sim U(0,1).\]

The problems (\ref{1.2A}), (\ref{1.2B}) were solved using the default {\tt cvxpy} solver for quadratic problems. The static and dynamic problems were modelled as explained above. For $m=15$, $n=25$, $c=10$, $\delta=0.1$ the results are shown in Fig.\,\ref{fig:2}

Qualitatively the results are similar to those shown in Fig.,\ref{fig:1}, but now the transfer price and the difference between the supply and demand are two-dimensional. In the static case (upper panel) the limiting prices are substantially different, since they are determined by random parameters of the revenue and cost functions, which are randomly fixed only once. In the dynamic case (lower panel) these parameters are sampled at each iteration from the same distribution, and the prices look similar, since we did not introduced any asymmetry in the underlying distributions.
\end{example}

\begin{figure}
\centering
\includegraphics[width=\textwidth]{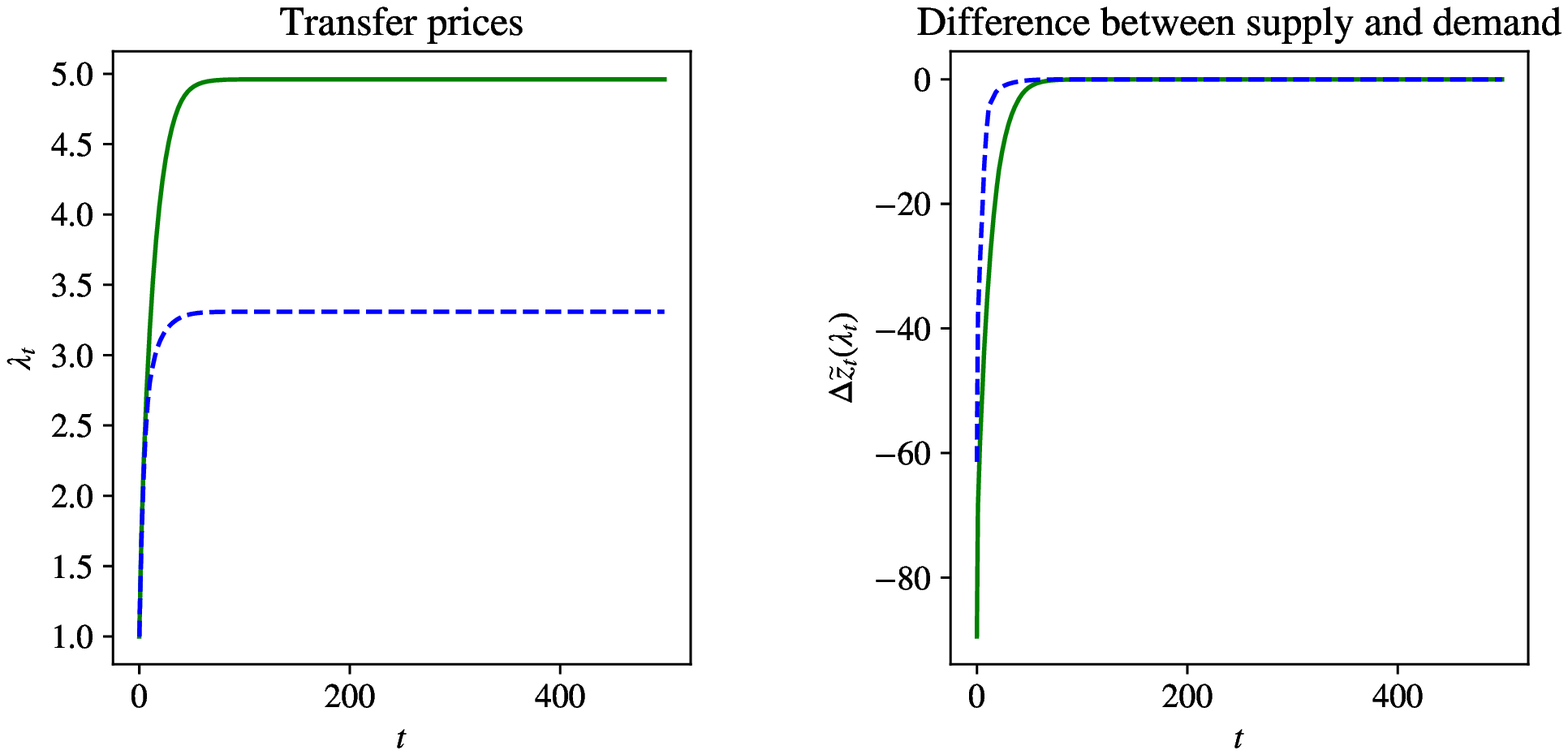}
\includegraphics[width=\textwidth]{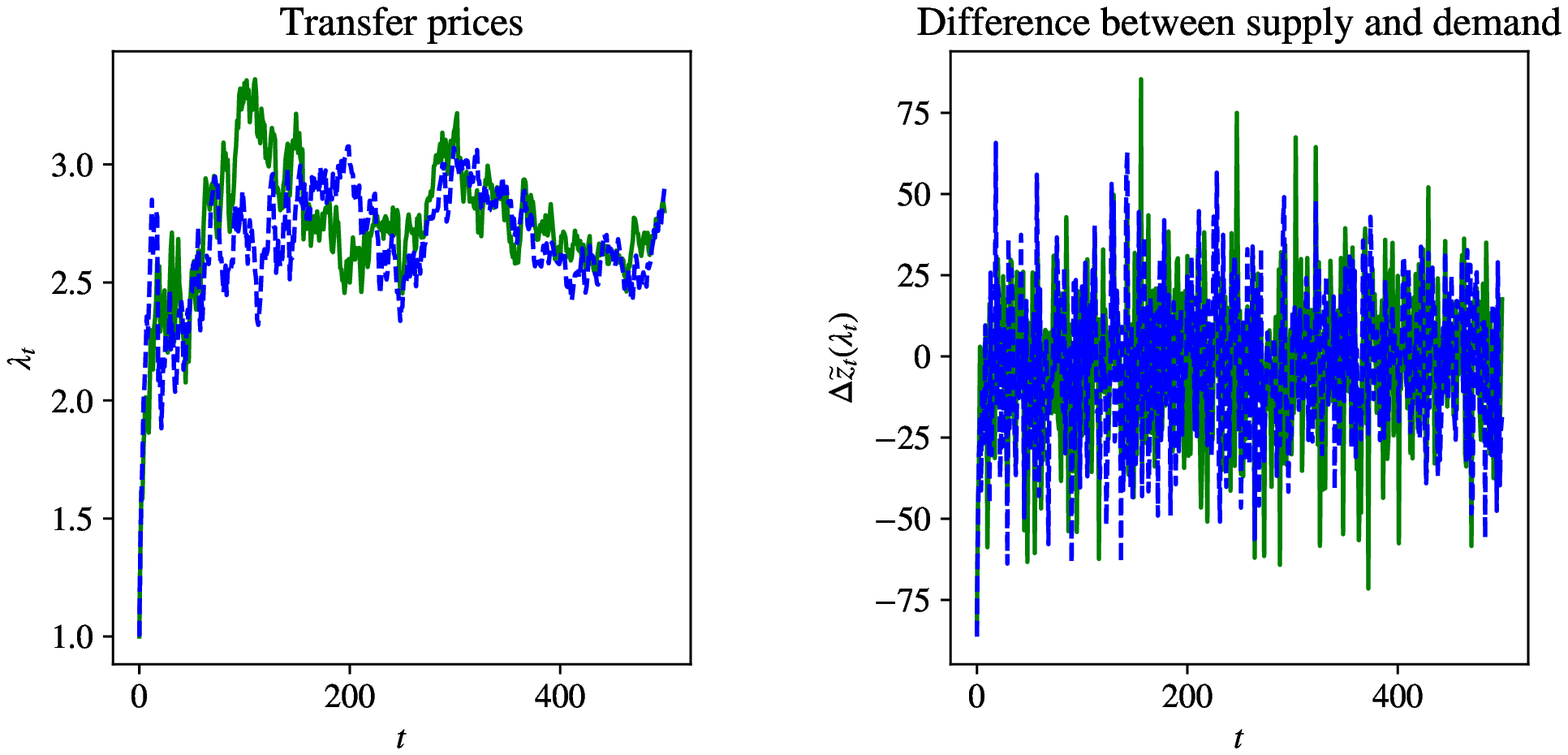}
\caption{Two commodities ($d=2$). The transfer prices and differences between the supply and demand in thte static (upper panel) and dynamic (lower panel) cases.}
 \label{fig:2}
\end{figure}

\section{Conclusion}
We considered a firm, consisting from multiple production and sales divisions, and presented a simple algorithm, producing a sequence of approximately firm-optimal transfer prices without any knowledge of the division revenue and cost functions.  This algorithm is an instance of the SOLO FTRL algorithm of \cite{Orabona2018}, applied to the dual optimization problem. The quantitative estimates are given in Theorems \ref{th:3}, \ref{th:4}.

Let us mention some issues that were completely ignored. Any of them can be a subject of future work.  
\begin{itemize}
\item External market prices. We assumed that trading at transfer prices is the only opportunity for the firm divisions. However, it is not unusual to assume that they can also trade the same commodities at an external market.
\item Information manipulation by the divisions. For instance, the production divisions can produce smaller than the stage-optimal amounts of commodities to raise their transfer prices. They can also cheat, when communicating their costs.
\item Inventory and unmet demand problems. At each time step the difference between supply and demand was a driving force for the transfer prices. In reality, the surplus  commodities should be stored or sold at an external market. Similarly, missing commodities need to be delivered to the sales divisions from a storage or an external market. We implicitly assumed that both operations are available, which need not be the case.
\end{itemize}

\begin{acknowledgement}
The research is supported by the Russian Science Foundation, project 17-19-01038. 
\end{acknowledgement}

\bibliographystyle{plain}      
\bibliography{rokh-ug_lit}

\end{document}